\documentclass{asl}

\usepackage{comment}
\usepackage[a4paper]{geometry}
\usepackage{times}
\usepackage[english]{babel}
\usepackage[latin1]{inputenc}
\usepackage{lmodern}
\usepackage[T1]{fontenc}
\usepackage{hyperref}

\usepackage{tikz} 

\theoremstyle{plain}
\newtheorem{theorem}{Theorem}[section]

\newtheorem{lemma}[theorem]{Lemma}
\newtheorem{corollary}[theorem]{Corollary}

\theoremstyle{definition}
\newtheorem{definition}[theorem]{Definition}
\newtheorem{notation}[theorem]{Notation}

\theoremstyle{remark}
\newtheorem{remark}[theorem]{Remark}

\newcommand{\SpectorEq}{\mathcal{BR}}
\newcommand{\GeneralEq}[1]{\mathcal{GBR}_{#1}}

\newcommand{\eqleft}[1]{\begin{itemize} \item[] $#1$ \end{itemize}}

\DeclareMathOperator{\AC}{{\bf AC}}
\DeclareMathOperator{\HAomega}{{\bf HA}^\omega}

\newcommand{\vgt}[1]{``#1''}
\newcommand{\NN}{\mathbb{N}}

\newcommand{\ap}[1]{\langle #1 \rangle}
\newcommand{\apc}[2]{\langle #1 ; #2 \rangle}
\newcommand{\bp}[1]{\left\lbrace #1 \right\rbrace}

\newcommand{\level}[1]{{\rm level}(#1)}

\newcommand{\BR}{{\sf BR}}

\newcommand{\T}{{\sf T}}
\newcommand{\Rec}{{\sf Rec}}
\newcommand{\Suc}{{\sf Succ}}

\newcommand{\eqdef}{\stackrel{{\rm def}}{=}}
\newcommand{\Val}{{\rm Val}}

\newcommand{\BRSec}{{\sf B}}

\newcommand{\hyde}[1] 

\begin{document}

\title{A Direct Proof of Schwichtenberg's \\ Bar Recursion Closure Theorem}

\author{Paulo Oliva and Silvia Steila}

\maketitle

\begin{abstract} In \cite{Schwichtenberg}, Schwichtenberg showed that the System $\T$ definable functionals are closed under a rule-like version Spector's bar recursion of lowest type levels $0$ and $1$. More precisely, if the functional $Y$ which controls the stopping condition of Spector's bar recursor is $\T$-definable, then the corresponding bar recursion of type levels $0$ and $1$ is already $\T$-definable. Schwichtenberg's original proof, however, relies on a detour through Tait's infinitary terms and the correspondence between ordinal recursion for $\alpha < \varepsilon_0$ and primitive recursion over finite types. This detour makes it hard to calculate on given concrete system $\T$ input, what the corresponding system $\T$ output would look like. In this paper we present an alternative (more direct) proof based on an explicit construction which we prove correct via a suitably defined logical relation. We show through an example how this gives a straightforward mechanism for converting bar recursive definitions into $\T$-definitions under the conditions of Schwichtenberg's theorem. Finally, with the explicit construction we can also easily state a sharper result: if $Y$ is in the fragment $\T_i$ then terms built from $\BR^{\NN, \sigma}$ for this particular $Y$ are definable in the fragment $\T_{i + \max \{ 1, \level{\sigma} \} + 2}$. 
\end{abstract}

\section{Introduction}

In \cite{Goedel}, G\"{o}del interpreted intuitionistic arithmetic in a quantifier-free type theory with primitive recursion in all finite types, the so-called System $\T$. This interpretation became known as \vgt{Dialectica}, the name of the journal where it was published. The Dialectica interpretation of arithmetic was extended by Spector to classical analysis in the system \vgt{$\T +$ bar recursion} \cite{Spector(62)}.

The schema of Spector's bar recursion (for a pair of finite types $\tau, \sigma$) is defined as
\begin{equation} \label{spector-def}
\BR^{\tau, \sigma}(G, H, Y)(s) \stackrel{\sigma}{=} 
\left\{
\begin{array}{ll}
	G(s) & {\rm if} \; Y(\hat s) < |s| \\[2mm]
	H(s)(\lambda x^\tau . \BR(G,H,Y)(s * x)) & {\rm otherwise}
\end{array}
\right.
\end{equation}
where $s \colon \tau^*$, $G \colon \tau^* \to \sigma$, $H \colon \tau^* \to (\tau \to \sigma) \to \sigma$ and $Y \colon (\NN \to \tau) \to \NN$. As usual $\hat s$ denotes the infinite extension of the finite sequence $s$ with $0$'s of appropriate type. For clarify of exposition we prefer to separate the arguments that stay fixed during the recursion, namely $G, H$ and $Y$, from the mutable argument $s$.

In \cite{Schwichtenberg}, Schwichtenberg proved that if $Y, G$ and $H$ are closed terms of system $\T$, and if $\tau$ is of type level $0$ or $1$, then the functional $\lambda s . \BR^{\tau, \sigma}(G, H, Y)(s)$ is already $\T$-definable.

Schwichtenberg's original proof is based on the notion of infinite terms as introduced by Tait \cite{Tait} and his argument requires the normalization theorem for infinite terms and the valuation functional provided in \cite{Schwichtenberg73}.  Schwichtenberg  proves that  bar recursions of type levels $0$ and  $1$ are reducible to $\alpha$-recursion for some $\alpha < \varepsilon_0$. Hence, using an interdefinability result from Tait \cite{Tait}, he concludes that they are also reducible to primitive recursions of higher types. Such detour makes it extremely difficult to work out the $\T$-definition of $\lambda G, H, s . \BR^{\tau, \sigma}(G, H, Y)(s)$ for a given concrete $\T$ definable $Y$, for instance, $Y(\alpha) = \Rec^\NN(0, \lambda k. \alpha)(\alpha(0))$, where $\alpha \colon \NN \to \NN$ and $k$ is a fresh variable.

In here we present a direct inductive proof of Schwichtenberg's result which provides an explicit method to eliminate bar recursion of type levels $0$ and $1$ when $Y$ is a concrete system $\T$ term. The focus of our result is syntactic: We describe an effective construction that given a term in $\T + \BR$, satisfying the above restrictions, will produce an equivalent term in system $\T$. We also strengthen Schwichtenberg's result by showing that when $Y$ is $\T$-definable and $\tau$ is of type level $0$ or $1$, then the functional $\lambda G, H, s . \BR^{\tau, \sigma}(G, H, Y)(s)$ is already $\T$-definable (uniformly in $G$ and $H$).

Our proof is composed of two main parts. In the first part (Section \ref{sec-general-br}) we define a variant of bar recursion which we call \emph{general bar recursion} -- a family of bar recursive functions parametrized by bar predicates. We show that when the bar predicate ``secures'' the functional $Y$, then $\BR$ for that $Y$ can be defined from the general bar recursion. In the second part (Section \ref{sec-main-result}) we present the main construction: Given a $\T$-definable $Y$, we can $\T$-define a general bar recursion for a bar predicate which secures $Y$.  The construction of the term which corresponds to the given $\T + \BR$ term is syntactic, as its definition is by induction on the structure of the input term. The proof of equivalence is carried out in intuitionistic Heyting arithmetic in all finite types $\HAomega$. One can, however, also view the result model-theoretically, by looking at models of $\HAomega$. Our result establishes that restricted bar-recursive terms have a denotation which falls within the subset of $\T$-definable elements.

\subsection{Spector's bar recursion} 

The finite types are defined inductively, where $\NN$ is the basic finite type, $\tau_0 \to \tau_1$ is the type of functions from $\tau_0$ to $\tau_1$, and $\tau_0^*$ is the type of finite sequences whose elements are of type $\tau_0$. Note that we have, for convenience, enriched the type system with the type of finite sequences. As usual, we often write $\tau_1^{\tau_0}$ for the type $\tau_0 \to \tau_1$.

System $\T$ \cite{Goedel,Spector(62)} consists of the simply typed $\lambda$-calculus with natural numbers ($0$ and $\Suc$) and the recursor $\Rec^\rho$, for each finite type $\rho$, together with the associated equations:
\begin{equation} 
\Rec^{\rho}(a, f)(n) \stackrel{\rho}{=} 
\left\{
\begin{array}{ll}
	a & {\rm if} \; n = 0 \\[2mm]
	f(m,\Rec^\rho(a, f)(m)) & {\rm if} \; n = \Suc(m)
\end{array}
\right.
\end{equation}
where $a \colon \rho$ and $f \colon \NN \to \rho \to \rho$. When translating bar recursive terms into system $\T$ terms we will also make use of a definitional extension of $\T$ with finite products $\tau \times \sigma$. When $s \colon \tau$ and $t \colon \sigma$ we write $\apc{s}{t}$ for the element of type $\tau \times \sigma$.

As usual, $\NN$ has type level $0$; the type level of $\rho \to \eta$ is the maximum between the type level of $\rho$ plus $1$  and the type level of $\eta$; the type level of $\tau \times \sigma$ is the maximum between the type level of $\tau$ and the type level of $\sigma$; the type level of $\tau^*$ is the type level of $\tau$. We write $\level{\tau}$ for the type level of $\tau$. The fragment of $\T$ where the recursor $\Rec^\rho$ is restricted to types $\rho$ with $\level{\rho} \leq i$ is denoted $\T_i$.

\begin{definition}[Spector's bar recursion] For each pair of types $\tau, \sigma$, let $\SpectorEq^{\tau, \sigma}$ be the universal formula
\begin{equation*} \label{spector-def-eq}
\SpectorEq^{\tau, \sigma}(\xi,G,H,Y) \; \eqdef \; \forall s^{\tau^*} 
\left\{
\begin{array}{lcl}
	Y(\hat s) < |s| & \to & \xi(G, H, Y)(s) \stackrel{\sigma}{=} G(s)  \\[1mm]
	& \wedge&   \\[1mm]
	Y(\hat s) \geq |s| & \to & \xi(G, H, Y)(s) \stackrel{\sigma}{=} H(s)(\lambda x^\tau . \xi(G,H,Y)(s * x))
\end{array}
\right\}
\end{equation*}
where
\[ \xi \colon (\tau^* \to \sigma) \to (\tau^* \to (\tau \to \sigma) \to \sigma) \to (\tau^\NN \to \NN) \to \tau^* \to \sigma \]
The extension of system $\T$ with Spector's bar recursion consists of adding to the language of $\T$ a family of constants $\BR^{\tau, \sigma}$, for each pair of finite types $\tau, \sigma$, together with the defining axioms $\forall G, H, Y \, \SpectorEq^{\tau,\sigma}(\BR^{\tau,\sigma},G,H,Y)$. We speak of Spector bar recursion of type level $i$ when $\tau$  has type level $i$.

\end{definition}


When we omit an argument of $\SpectorEq^{\tau, \sigma}(\xi,G,H,Y)$ we will assume it is universally quantified, e.g. 
\[
\begin{array}{lcl}
\SpectorEq^{\tau,\sigma}(\xi, G) & \eqdef & \forall H, Y \, \SpectorEq^{\tau,\sigma}(\xi,G,H,Y) \\[2mm]
\SpectorEq^{\tau,\sigma}(\xi) & \eqdef & \forall G, H, Y \, \SpectorEq^{\tau,\sigma}(\xi,G,H,Y)
\end{array}
\]
We will also use \emph{named parameters} in order to fix a particular parameter of $\xi$, e.g. if $t$ is a term having the same type as $Y$ then $\SpectorEq^{\tau,\sigma}(\xi, Y=t)$ stands for the formula
\[
\forall G, H \, \forall s^{\tau^*} 
\left\{
\begin{array}{lcl}
	t(\hat s) < |s| & \to & \xi(G, H)(s) \stackrel{\sigma}{=} G(s)  \\[1mm]
	& \wedge&   \\[1mm]
	t(\hat s) \geq |s| & \to & \xi(G, H)(s) \stackrel{\sigma}{=} H(s)(\lambda x^\tau . \xi(G,H)(s * x))
\end{array}
\right\}
\]
where we replace $Y$ by $t$ and omit the argument $Y$ from $\xi$. Finally, when clear from the context we will omit the superscript types, and write simply $\SpectorEq$.

\begin{remark}[Related work] A previous analysis by Kreisel (see e.g. \cite{Spector(62)}), together with the reduction provided by Howard \cite{Howard1}, guarantees that system $\T$ is not closed under the bar recursion rule when $\tau$ has type level greater or equal to $2$.  Diller \cite{Diller} presented a reduction of bar recursion  to $\alpha$-recursion for some bounded ordinal $\alpha$, while Howard \cite{Howard2, Howard3} provided an ordinal analysis of the constant of bar recursion of type level $0$. Kreuzer \cite{Kreuzer12} refined Howard's ordinal analysis of bar recursion in terms of Grzegorczyk's hierarchy. In \cite{Kohlenbach99}, Kohlenbach generalised Schwichtenberg's result by showing that if $Y[\vec{x},\vec{f}] \colon \NN$ is a term with variables $\vec{x}$ of type level $0$ and $\vec{f}$ of type level $1$, then the bar recursive functional of type level $0$ provided by $Y$, is $\T$-definable. Kohlenbach's argument is based on the observation that in Schwichtenberg's result no restrictions are put on the type $\sigma$, hence it is possible to relativize Schwichtenberg's proof where $Y$ is allowed to contains parameters of type levels $0$ and $1$ in system $\T$. The same argument can be carried over to our construction below.
\end{remark}

\begin{notation} Throughout the paper we adopt the following conventions:
\begin{itemize}
\item We use $\tau, \sigma, \rho, \eta$ to denote finite types.
\item We write $a \colon \tau$ or $a^\tau$ to indicate that $a$ is a term of type $\tau$. 
\item A tuple of variables $x_1, \ldots, x_n$ will be denoted by $\vec{x}$.
\item The term $0^\tau$  denotes the standard inductively defined zero object of type $\tau$.
\item Given a finite sequence $s \colon \tau^*$, $\hat{s} \colon \NN \to \tau$ denotes the extension of $s$ with infinitely many $0^\tau$. 
\item For any finite sequence $s \colon \tau^*$ and any $x \colon \tau$, $s*x$ denotes appending $x$ to $s$.
\item For any finite sequences $s, s' \colon \tau^*$, $s*s'$ denotes their concatenation. 
\item Given $s \colon \tau^*$ and an infinite sequence $\alpha \colon \tau^\NN$, we also write $s*\alpha$ to denote their concatenation. 
\item For any infinite sequence $\alpha \colon \NN \to \tau$, $\bar{\alpha}n$ denotes the finite sequence $\ap{\alpha(0), \dots, \alpha(n-1)}$. We also use the same notation for finite sequences $s \colon \tau^*$ when $n \leq |s|$. 
\end{itemize}
\end{notation}

\section{General Bar Recursion}
\label{sec-general-br}

Let us start by observing that if $Y \colon (\NN \to \NN) \to \NN$ is a constant function then bar recursion for such $Y$ is $\T$-definable, for any types $\tau, \sigma$.

\begin{lemma}[$\HAomega$] \label{constant-lemma} For each $\tau, \sigma$, let $i = \max\{1 + \level{\tau}, \level{\sigma}\}$, there is a closed term
\[ \Psi \colon \NN \to (\tau^* \to \sigma) \to (\tau^* \to (\tau \to \sigma) \to \sigma) \to \tau^* \to \sigma \]
in $\T_i$ such that for all $k \colon \NN$ we have $\SpectorEq(\Psi(k), Y = \lambda \alpha . k)$.
\end{lemma}
\begin{proof} We define a term $\Psi$ and show it satisfies
\[ 
\Psi(k)(G, H)(s) \stackrel{\sigma}{=} 
\left\{
\begin{array}{ll}
	G(s) & {\rm if} \; |s| > k \\[2mm]
	H(s)(\lambda x^\tau . \Psi(k)(G,H)(s * x)) & {\rm if} \; |s| \leq k
\end{array}
\right.
\]
for all $k, G, H$ and $s$. First define the functional
\[
\varphi \colon (\tau^* \to \sigma) \to (\tau^* \to (\tau \to \sigma) \to \sigma) \to \NN \to \tau^* \to \sigma
\]
by primitive recursion as
\begin{equation} \label{varphi-def}
\varphi(G, H)(n) \eqdef  
\left\{
\begin{array}{ll}
	G & {\rm if} \; n = 0 \\[2mm]
	\lambda s^{\tau^*} . H(s)(\lambda x^\tau . \varphi(G,H)(n-1)(s * x)) & {\rm if} \; n > 0.
\end{array}
\right.
\end{equation}
Then, using $\varphi$, define the functional $\Psi$ by cases as
\begin{equation} \label{Psi-def}
\Psi(k)(G, H)(s) \eqdef 
\left\{
\begin{array}{ll}
	G(s) & {\rm if} \; |s| > k \\[2mm]
	\varphi(G, H)(k + 1 - |s|)(s) & {\rm if} \; |s| \leq k.
\end{array}
\right.
\end{equation}
Clearly the functional $\Psi$ is $\T$-definable, and only requires primitive recursion of type $\tau^* \to \sigma$, so it is in fact definable in $\T_i$ for $i = \max\{1 + \level{\tau}, \level{\sigma}\}$. It remains for us to prove that $\Psi(k)(G,H)$ satisfies the above mentioned equation. \\[1mm]
Let $k, G, H$ and $s$ be fixed. 
If $|s| > k$ then
\[ \Psi(k)(G, H)(s) \stackrel{(\ref{Psi-def})}{=} G(s). \]
When $|s| \leq k$, we distinguish two cases. If $(\dagger) \, |s| = k$ then
\[
\begin{array}{lcl}
\Psi(k)(G,H)(s)
	& \stackrel{(\ref{Psi-def})}{=} & \varphi(G,H)(k+1-|s|)(s) \\
	& \stackrel{(\dagger)}{=} & \varphi(G,H)(1)(s) \\
	& \stackrel{(\ref{varphi-def})}{=} & H(s)(\lambda x^\tau . \varphi(G,H)(0)(s * x)) \\
	& \stackrel{(\ref{varphi-def})}{=} & H(s)(\lambda x^\tau . G(s * x)) \\
	& \stackrel{(\ref{Psi-def})}{=} & H(s)(\lambda x^\tau . \Psi(k)(G,H)(s * x)).
\end{array}
\]
If $|s| < k$, we have
\[
\begin{array}{lcl}
\Psi(k)(G,H)(s)
	& \stackrel{(\ref{Psi-def})}{=} & \varphi(G,H)(k+1-|s|)(s) \\
	& \stackrel{(\ref{varphi-def})}{=} & H(s)(\lambda x. \varphi(G,H)(k+1-|s|-1)(s*x)) \\[1mm]
	& = & H(s)(\lambda x. \varphi(G,H)(k+1-|s*x|)(s*x)) \\
	& \stackrel{(\ref{Psi-def})}{=} & H(s)(\lambda x. \Psi(k)(G,H)(s*x)).
\end{array}
\]
\end{proof}

A predicate $S(s^{\tau^*})$ is called a \emph{bar} if it satisfies the following three conditions:
\begin{itemize}
	\item[$(i)$] \emph{Decidable}: $\forall s^{\tau^*} (S(s) \vee \neg S(s))$
	\item[$(ii)$] \emph{Bar}: $\forall \alpha^{\tau^\NN} \exists n^\NN S(\bar{\alpha} n)$
	\item[$(iii)$] \emph{Monotone}: $\forall s^{\tau^*}, t^{\tau^*} (S(s) \to S(s * t))$
\end{itemize}
We now introduce a variant of Spector's bar recursion, which we call \emph{general bar recursion}. These are parametrized by a bar predicate $S(s^{\tau^*})$.

\begin{definition}[General bar recursion] For each pair of types $\tau, \sigma$, and a bar predicate $S(s^{\tau^*})$, let $\GeneralEq{S}^{\tau, \sigma}$ be the formula
\begin{equation} \label{secure-def}
\GeneralEq{S}^{\tau, \sigma}(\xi,G,H) \; \eqdef \; \forall s^{\tau^*} 
\left\{
\begin{array}{lcl}
	S(s) & \to & \xi(G, H)(s) \stackrel{\sigma}{=} G(s)  \\[1mm]
	& \wedge&   \\[1mm]
	\neg S(s) & \to & \xi(G, H)(s) \stackrel{\sigma}{=} H(s)(\lambda x^\tau . \xi(G,H)(s * x))
\end{array}
\right\}
\end{equation}
where $\xi \colon (\tau^* \to \sigma) \to (\tau^* \to (\tau \to \sigma) \to \sigma) \to \tau^* \to \sigma$.
\end{definition}

When clear from the context we will omit the superscript types, writing simply $\GeneralEq{S}$ instead of $\GeneralEq{S}^{\tau,\sigma}$. And once again, we write  $\GeneralEq{S}(\xi)$ as a shorthand for  $\forall G, H \, \GeneralEq{S}(\xi,G,H)$. 

\begin{definition} We say that a bar $S$ \emph{secures $Y \colon \tau^\NN \to \NN$} if for all $s^{\tau^*}$
\[ S(s) \quad \Rightarrow \quad \mbox{$\lambda \beta. Y(s*\beta)$ is constant.} \]
\end{definition}

\begin{theorem}[$\HAomega$] \label{secure-br-thm} Let $\tau, \sigma$ be fixed, and $i = \max\{1 + \level{\tau}, \level{\sigma}\}$. Let also $t \colon \tau^\NN \to \NN$ be a fixed closed term in $\T_i$. There is a $\T_i$-term $\Phi^t$ such that for any bar $S$ securing $t$
\[ \GeneralEq{S}(\Delta) \quad \Rightarrow \quad \SpectorEq(\Phi^t(\Delta), Y = t) \]
\end{theorem}

\begin{proof} Let $t$ be fixed and assume ($\dagger$) $S$ is a bar securing $t$. First, define the construction
\[ {\mathcal H}^t \colon (\tau^* \to \sigma) \to (\tau^* \to (\tau \to \sigma) \to \sigma) \to \tau^* \to (\tau \to \sigma) \to \sigma \]
as
\begin{equation} \label{H-eq}
{\mathcal H}^t(G,H)(s)(f^{\tau \to \sigma}) \eqdef
\begin{cases}
    G(s) & {\rm if} \; t(\hat s) < |s| \\[2mm]
    H(s)(f) & {\rm otherwise},
\end{cases}
\end{equation}
and let $\Phi^t$ be the $\T_i$-definable term:
\begin{equation} \label{Phi-Y-eq}
\Phi^t(\Delta)(G, H)(s) \eqdef \Delta(\lambda s' . \Psi(t(\widehat{s'}))(G, H)(s'), {\mathcal H}^t(G, H))(s)
\end{equation}
where $\Psi$ is the construction given in the proof of Lemma \ref{constant-lemma} and $\Delta$ has type $$(\tau^* \to \sigma) \to (\tau^* \to (\tau \to \sigma) \to \sigma) \to \tau^* \to \sigma.$$
Suppose $\Delta$ is such that $(\ddagger) \; \GeneralEq{S}(\Delta)$. We must show $\SpectorEq(\Phi^t(\Delta),Y = t)$.
First we must show that if $t(\hat s) < |s|$ then $\Phi^t(\Delta) = G(s)$. So we assume $t(\hat s) < |s|$ and consider two cases (using the decidability of the bar): \\[1mm]
If $S(s)$ then
\eqleft{
\begin{array}{lcl}
\Phi^t(\Delta)(G, H)(s)
	& \stackrel{(\ref{Phi-Y-eq})}{=} & \Delta(\lambda s' . \Psi(t(\widehat{s'}))(G, H)(s'), {\mathcal H}^t(G, H))(s) \\
	& \stackrel{(\ddagger)}{=} & \Psi(t \hat{s})(G, H)(s) \\
	& \stackrel{(\textup{L}\ref{constant-lemma})}{=} & G(s)
\end{array}
}
whereas if $\neg S(s)$, then
\eqleft{
\begin{array}{lcl}
\Phi^t(\Delta)(G, H)(s)
	& \stackrel{(\ref{Phi-Y-eq})}{=} & \Delta(\lambda s' . \Psi(t(\widehat{s'}))(G, H)(s'), {\mathcal H}^t(G, H))(s) \\
	& \stackrel{(\ddagger)}{=} & {\mathcal H}^t(G, H)(s)(\lambda x . \Phi^t(\Delta)(G,H)(s*x)) \\
	& \stackrel{(\ref{H-eq})}{=} & G(s).
\end{array}
}
Secondly, we must show that when $t(\hat s) \geq |s|$ then $$\Phi^t(\Delta)(G,H)(s) = H(s)(\lambda x . \Phi^t(\Delta)(G,H)(s * x)).$$ 
Again we assume $t(\hat s) \geq |s|$ and consider two cases: \\[1mm]
If $S(s)$ then, by our assumption ($\dagger$), $\lambda \beta. t(s*\beta)$ is constant, and in particular $(*)$ $t(\widehat{s*x}) = t(\hat{s})$. By monotonicity of the bar we also have $S(s * x)$. Hence
\eqleft{
\begin{array}{lcl}
\Phi^t(\Delta)(G,H)(s)
	& \stackrel{(\ref{Phi-Y-eq})}{=} & \Delta(\lambda s' . \Psi(t(\widehat{s'}))(G, H)(s), {\mathcal H}^t(G, H))(s) \\
	& \stackrel{(\ddagger)}{=} & \Psi(t(\hat{s}))(G, H)(s) \\
	& \stackrel{\textup{L}\ref{constant-lemma}}{=} & H(s)(\lambda x . \Psi(t(\hat{s}))(G, H)(s * x)) \\
	& \stackrel{(*)}{=} & H(s)(\lambda x . \Psi(t(\widehat{s*x}))(G, H)(s * x)) \\
	& \stackrel{(\ddagger)}{=} & H(s)(\lambda x . \Delta(\lambda s' . \Psi(t(\widehat{s'}))(G, H)(s'), {\mathcal H}^t(G, H))(s*x)) \\
	& \stackrel{(\ref{Phi-Y-eq})}{=} & H(s)(\lambda x . \Phi^t(\Delta)(G,H)(s * x)). 
\end{array}
}
Otherwise, if $\neg S(s)$ then
\eqleft{
\begin{array}{lcl}
\Phi^t(\Delta)(G,H)(s)
	& \stackrel{(\ref{Phi-Y-eq})}{=} & \Delta(\lambda s' . \Psi(t(\widehat{s'}))(G, H)(s'), {\mathcal H}^t(G, H))(s) \\
	& \stackrel{(\ddagger)}{=} & {\mathcal H}^t(G,H)(s)(\lambda x . \Delta(\lambda s' . \Psi(t(\widehat{s'}))(G, H)(s'), {\mathcal H}^t(G, H))(s*x)) \\
	& \stackrel{(\ref{Phi-Y-eq})}{=} & {\mathcal H}^t(G,H)(s)(\lambda x . \Phi^t(\Delta)(G,H)(s*x)) \\
	& \stackrel{(\ref{H-eq})}{=} & H(s)(\lambda x . \Phi^t(\Delta)(G,H)(s*x)).
\end{array}
}
\end{proof}

\section{Main Result}
\label{sec-main-result}

We have just shown that Spector's bar recursion, when $Y$ is a fixed $\T$-term $t$, is $\T$-definable in the general bar recursion for any predicate $S$ securing $t$. We will now prove that for $\tau=\NN$ or $\tau= \NN \to \NN$ and for any fixed term $t[\alpha]$, there exists some predicate $S$ securing the closed term $\lambda \alpha . t[\alpha]$ such that there is a $\T$-definable functional which satisfies the general bar recursion equation $\GeneralEq{S}$. For the rest of the section, let $\tau$ and $\sigma$ be fixed.

\begin{definition}  For each finite type $\eta$ we associate inductively a new finite type $\eta^\circ$ as:
	\[
	\begin{array}{lcl}
	\NN^\circ & = & (\tau^\NN \to \NN) \times ((\tau^* \to \sigma) \to (\tau^* \to (\tau \to \sigma) \to \sigma) \to \tau^* \to \sigma) \\[2mm]
	(\rho_0 \to \rho_1)^\circ & = & \rho_0^\circ  \to \rho_1^\circ
	\end{array}
	\]
\end{definition}

Since terms $t$ of type $\NN^\circ$ in fact consist of a pair of functionals, we will use the terminology
\begin{itemize}
	\item $\Val_t \colon \tau^\NN \to \NN$ for the first component of $t$, and
	\item $\BRSec_t \colon (\tau^* \to \sigma) \to (\tau^* \to (\tau \to \sigma) \to \sigma) \to \tau^* \to \sigma$ for the second component,
\end{itemize}
so that $t = \apc{\Val_t}{\BRSec_t}$.


\begin{lemma} \label{type-level-lemma} $\level{\eta^\circ} = 2 + \max \{ 1 + \level{\tau}, \level{\sigma} \} + \level{\eta}$. 
\end{lemma}

\begin{proof} By induction on the structure of $\eta$. 
	\begin{itemize}
		\item $\eta = \NN$. First notice that the type level of $\NN^\circ$ is dictated by the component $\tau^* \to (\tau \to \sigma) \to \sigma$. Since $$\level{\tau^* \to (\tau \to \sigma) \to \sigma} = \max \{ 2 + \level{\tau}, 1 + \level{\sigma} \}$$ we have that
		\[ \level{(\tau^* \to \sigma) \to (\tau^* \to (\tau \to \sigma) \to \sigma) \to \tau^* \to \sigma} = 2 + \max \{ 1 + \level{\tau}, \level{\sigma} \} \]
		so $\level{\NN^\circ} = 2 + \max \{ 1 + \level{\tau}, \level{\sigma} \} + \level{\NN}$.
		\item $\eta= \rho_0 \to \rho_1$. By definition
		$\level{\eta^{\circ}} = \max\{ 1 + \level{\rho_0^{\circ}}, \level{\rho_1^{\circ}}\}$. 
		By induction hypothesis for $i=0,1$ we have
		$$\level{\rho_i^{\circ}} = 2 + \max \{ 1+\level{\tau} , \level{\sigma} \} + \level{\rho_i}.$$
		Therefore
		\[
		\begin{array}{lcl}
		\level{\eta^{\circ}} & = & \max\{ 1 + \level{\rho_0^{\circ}}, \level{\rho_1^{\circ}}\}  \\
		& = & 2 + \max \{ 1+\level{\tau} , \level{\sigma} \}  + \max\{1+ \level{\rho_0}, \level{\rho_1}\}\\
		& = & 2 + \max \{ 1+\level{\tau} , \level{\sigma} \}  + \level{\eta}.
		\end{array}
		\]
	\end{itemize}
\end{proof}
\subsection{Translation (case $\tau = \NN$)}

For the rest of this sub-section we shall also assume that $\tau = \NN$, and that $\alpha$ is a special variable of type $\NN \to \NN$. In Section \ref{sec-case-fct} we describe which small changes need to be made to treat the case $\tau = \NN \to \NN$. Moreover, we assume $\sigma$ to be an arbitrary but fixed finite type.

Given a term $t \colon \NN$ with the special variable $\alpha$ as the only free variable, our goal is to define a term $t^\circ \colon \NN^\circ$ in such a way that $\Val_{t^\circ} = \lambda \alpha . t$, allowing us to evaluate $t$ for concrete values of $\alpha$, and $\BRSec_{t^\circ}$ will be such that $\GeneralEq{S}(\BRSec_{t^\circ})$, for some bar $S$ which secures $\lambda \alpha . t$. For a term $t$ of a higher-type we will define $t^\circ$ in such a way that this property is preserved at ground type. 

\begin{definition} \label{circ-def} Let $\Psi(k)$ be the $\T$-term defined in the proof of Lemma \ref{constant-lemma} (defining bar recursion in the special case when $Y$ is the constant functional $\lambda \alpha . k$). Assume a given mapping of variables $x \colon \eta$ to variables $x^\circ \colon \eta^\circ$, and let $\alpha$ be a special variable of type $\NN \to \tau$, where in this section $\tau$ is assume to be $\NN$. For any term $t \colon \rho$ in system $\T$, define $t^\circ \colon \rho^\circ$ inductively as follows:
\[
\begin{array}{lcl}
	0^\circ & \eqdef & \apc{\lambda \alpha. 0}{\lambda G,H.G} \\[2mm]
	\Suc^\circ & \eqdef &  \lambda x^{\NN^\circ} . \apc{\lambda \alpha . \Suc(\Val_x(\alpha))}{\BRSec_x}\\[2mm]
	\alpha^\circ & \eqdef &  \lambda x^{\NN^\circ} . \apc{\lambda \alpha . \alpha(\Val_x(\alpha))}{\lambda G, H. \BRSec_x(\lambda s' . \Psi(\Val_x(\widehat{s'}))(G,H)(s'),H)} \\[2mm]
	%
	(x^\eta)^\circ & \eqdef & x^\circ \\[2mm]
	(\lambda x^\eta . t)^\circ & \eqdef &\lambda x^\circ . t^\circ \\[2mm]
	(u v)^\circ & \eqdef & u^\circ v^\circ  \\[2mm]
	(\Rec^\eta)^\circ & \eqdef & \lambda a^{\eta^\circ}, F^{\NN^\circ \to \eta^\circ \to \eta^\circ}, x^{\NN^\circ}, v^{\rho^\circ} . \apc{\lambda \alpha . \Val_{r[\Val_x(\alpha)]}(\alpha)}{{\sf B}}
\end{array}
\]
where in the case of the $\Rec^{\eta}$ we assume $\eta = \rho \to \NN$, and $r[n]$ and ${\sf B}$ are built from $a, F, x$ and $v$ as
\begin{itemize}
	\item $r[n] \eqdef \Rec^{\eta^\circ}(a, \lambda k^\NN . F(k^\circ))(n)(v)$
	\item ${\sf B}(G,H)(s) \eqdef \BRSec_x(\lambda s' . \BRSec_{r[\Val_x(\widehat{s'})]}(G,H)(s'),H)(s)$
\end{itemize}
using the abbreviation $k^\circ \eqdef \apc{\lambda \alpha. k}{\lambda G, H .G}$ in the definition of $r[n]$. If $\eta = \NN$ then we may omit $\rho$ and the variable $v^{\rho^\circ}$, and should define $r[n] \eqdef \Rec^{\eta^\circ}(a, \lambda k^\NN . F(k^\circ))(n)$.
\end{definition}

Note that if $t:\NN$ has variables $\alpha$ and $x_1,\ldots,x_n$ free, then $t^\circ$ will only have $x_1^\circ, \ldots, x_n^\circ$ free.

\subsection{Verification}

We will now show that for any term $t[\alpha] \colon \NN$, the second component of $(t[\alpha])^\circ$, i.e. $\BRSec_{(t[\alpha])^\circ}$, is a term in system $\T$ which defines a general bar recursion for some bar predicate $S$ which secures $\lambda \alpha . t[\alpha]$.

\begin{theorem}[$\HAomega + \AC_0$] \label{main-theorem} Let $\tau = \NN$ and $t \colon \NN$ be a term of system $\T$ with only $\alpha^{\NN \to \tau}$ as free variable. Then there exists a bar $S$ which secures $\lambda \alpha . t$ such that $\GeneralEq{S}(\BRSec_{t^\circ})$.
\end{theorem}

\begin{proof} 
Let $\sim_\rho \, \subseteq \, \rho^\circ \times (\NN^\NN \to \rho)$ be the logical relation between terms of system $\T$ defined as:
\[
\begin{array}{ccc}
f^{\NN^\circ} \sim_{\NN} g^{\NN^\NN \to \NN} & \eqdef & \Val_f = g \,\wedge\, \exists S (S \mbox{ is a bar securing } g \mbox{ and }  \GeneralEq{S}(\BRSec_f)) \\[2mm]
f^{\rho_0^\circ \to \rho_1^\circ} \sim_{\rho_0 \to \rho_1} g^{\NN^\NN \to (\rho_0 \to \rho_1)} & \eqdef &
    \forall x^{\rho_0^\circ} \forall y^{\NN^\NN \to \rho_0} (x \sim_{\rho_0} y \implies f(x) \sim_{\rho_1}  \lambda \alpha . g(\alpha)(y \alpha))
\end{array}
\]
We prove that for any $t$ with free variables $\vec{x}$ and (possibly) $\alpha$, $(\lambda \vec{x}.t)^\circ \sim_\rho \lambda \alpha \lambda \vec{x} . t$ by structural induction over $t$, where $\rho$ is the type of $\lambda \vec{x}.t$.
\begin{itemize}
\item $t = 0$. We need to show that
\[
0^\circ = \apc{\lambda \alpha. 0}{\lambda G,H, s . G(s)} \sim_{\NN}  \lambda \alpha. 0.
\]
Clearly we have $\Val_{0^\circ} = \lambda \alpha . 0$. Let $S(s) \eqdef {\sf true}$, which is a bar securing $\lambda \alpha . 0$. Then, indeed we also have $\GeneralEq{S}(\BRSec_{0^\circ})$ since
\[
\BRSec_{0^\circ}(G,H)(s) = G(s).
\]
\item $t= \Suc$. Let us show that $\Suc^\circ \sim_{\NN \to \NN} \lambda \alpha . \Suc$, i.e. for all $x \colon \NN^\circ$ and $g^{\NN^\NN \to \NN}$
\[ x \sim_\NN g \quad \implies \quad \Suc^\circ(x) \sim_\NN  \lambda \alpha . \Suc(g \alpha) \]
The premise ensures that $\Val_x = g$ and $\GeneralEq{S_x}(\BRSec_x)$ for some bar $S_x$ securing $g$. Hence, assuming the premise, and unfolding the definition of $\Suc^\circ$, we need to show
\[ \apc{\lambda \alpha . \Suc(g \alpha)}{\BRSec_x} \sim_\NN  \lambda \alpha . \Suc(g \alpha). \]
The only non-trivial part is to observe that if $S_x$ secures $g$ then it also secures $\lambda \alpha . \Suc(g \alpha)$.
\item $t = z^\rho$. When $t$ is simply a free-variable $z$ we must show that $(\lambda z . z)^\circ \sim_{\rho \to \rho} \lambda \alpha \lambda z . z$. But this follows directly from the definition of $\sim_{\rho \to \rho}$, noticing that $(\lambda z^\rho . z)^\circ \eqdef \lambda z^{\rho^\circ} . z$
\item $t = \alpha$. We need to show that $\alpha^\circ \sim_{\NN \to \NN} \lambda \alpha . \alpha$, i.e. for all $x^{\NN^\circ}$ and $g^{\NN^\NN \to \NN}$
\[ x \sim_\NN g \quad \implies \quad \alpha^\circ(x) \sim_\NN  \lambda \alpha . \alpha(g \alpha). \]
Again, the premise $x \sim_\NN g$ implies that $\Val_x = g$ and $\GeneralEq{S_x}(\BRSec_x)$, for some bar $S_x$ securing $g$. Hence, fix $x$ and $g$ such that $x \sim_\NN g$. Unfolding the definition of $\alpha^\circ$, we show
\[  \apc{\lambda \alpha . \alpha (g \alpha)}{\lambda G,H, s . \BRSec_x(\lambda s' . \Psi(g(\widehat{s'}))(G, H)(s'),H)(s)} \sim_\NN  \lambda \alpha . \alpha(g \alpha) \]
where $\Psi(g(\widehat{s'}))(G,H)$ is the $\T$-definition of $\BR^{\tau, \sigma}(G,H,\lambda \alpha. g(\widehat{s'}))$ (cf. Lemma \ref{constant-lemma}). 
The first conjunct of the definition of $\sim_\NN$ is trivially satisfied. Let $$S(s) \eqdef S_x(s) \wedge g \hat{s} < |s|.$$ Since $S_x(s)$ is a bar, and $S_x$ secures $g$, it follows that $S(s)$ is also a bar. Moreover, since $S_x$ secures $g$, it also follows that $S$ secures $\lambda \alpha . \alpha(g \alpha)$. Using the hypothesis $(\dagger) \, \GeneralEq{S_x}(\BRSec_x)$, we need to show 
\[ \GeneralEq{S}(\lambda G, H. \BRSec_x(\lambda s' . \Psi(g(\widehat{s'}))(G,H)(s'),H)). \]
Fix $G, H$ and $s$. Consider two cases: \\[1mm]
If $S(s)$ then $S_x(s)$ and $g \hat s < |s|$. In this case we trivially have
\[ \BRSec_x(\lambda s' . \Psi(g(\widehat{s'}))(G,H)(s'),H)(s) \stackrel{(\dagger)}{=} \Psi(g(\hat{s}))(G, H)(s) = G(s) \]
If $\neg S(s)$ then either $\neg S_x(s)$ or $g \hat s \geq |s|$. We consider two cases: \\[1mm]
If $S_x(s)$ holds then $g \hat s \geq |s|$. 
Moreover, $(\ddagger) \, g \hat{s} = g (\widehat{s * y})$ for any $y$, since $S_x$ secures $g$. By monotonicity of $S_x$ we also have $S_x(s*y)$ for any $y$.  
Hence
\begin{align*}
\BRSec_x(\lambda s' . \Psi(g(\widehat{s'}))(G, H)(s'),H)(s)
    &\stackrel{(\dagger)}{=} \Psi(g(\hat{s}))(G, H)(s) \\[1mm]
    &= H(s)(\lambda y . \Psi(g(\hat{s}))(G, H)(s * y)) \\
    &\stackrel{(\ddagger)}{=} H(s)(\lambda y . \Psi(g(\widehat{s*y}))(G, H)(s * y)) \\
    &\stackrel{(\dagger)}{=} H(s)(\lambda y . \BRSec_x(\lambda s' . \Psi(g(\widehat{s'}))(G, H)(s'),H)(s*y))
\end{align*}
If $\neg S_x(s)$ then
\begin{align*}
\BRSec_x(\lambda s' . \Psi(g(\widehat{s'}))(G, H)(s'),H)(s)
    &\stackrel{(\dagger)}{=} H(s)(\lambda y . \BRSec_x(\lambda s' . \Psi(g(\widehat{s'}))(G, H)(s'),H)(s*y))
\end{align*}
\item $t = \lambda x^\rho. u$. Trivial by induction hypothesis. 
\item $t = u^{\rho \to \tau} v^\rho$. For simplicity let us assume $u$ has free-variables $x_1^{\sigma_1}$ and $x_2^{\sigma_2}$ and $v$ has free-variables $x_2^{\sigma_2}$ and $x_3^{\sigma_3}$, which is enough to illustrate how the difference in the set of free-variables of $u$ and $v$ is handled. We must show that 
$$(\lambda x_1,x_2,x_3 . u v)^\circ \sim_{(\sigma_1\times\sigma_2\times\sigma_3) \to \tau} \lambda \alpha \lambda x_1,x_2,x_3. u v.$$ 
By the definition of $(\cdot)^\circ$ this is
\[  \lambda x_1^\circ, x_2^\circ, x_3^\circ. u^\circ v^\circ \sim_{(\sigma_1\times\sigma_2\times\sigma_3) \to \tau} \lambda \alpha \lambda x_1,x_2,x_3 . u v \]
By induction hypothesis we have $\lambda x_1^\circ, x_2^\circ. u^\circ \sim_{(\sigma_1\times \sigma_2) \to (\rho \to \tau)} \lambda \alpha, x_1, x_2. u$, i.e. for all $x_1^\circ$,$\tilde{x}_1^{\NN^\NN \to \sigma_1}$, $x_2^{\circ}$,$\tilde{x}_2^{\NN^\NN \to \sigma_2}$ and $y^{\circ}, \tilde{y}^{\NN^\NN \to \rho}$
\[
x_1^\circ \sim_{\sigma_1} \tilde{x}_1 \wedge x_2^\circ \sim_{\sigma_2} \tilde{x}_2 \wedge  y^\circ \sim_\rho \tilde{y} 
\quad \implies \quad u^\circ y^\circ \sim_{\tau}  \lambda \alpha . u[\tilde{x}_1 \alpha/x_1][\tilde{x}_2 \alpha/x_2](\tilde{y} \alpha),
\]
and $\lambda x_2^\circ,x_3^\circ. v^\circ \sim_{(\sigma_2 \times \sigma_3) \to \rho} \lambda \alpha, x_2,x_3. v$, i.e. for all $x_2^\circ$, $\tilde{x}_2^{\NN^\NN \to \sigma_2}$, $x_3^\circ$ and $\tilde{x}_2^{\NN^\NN \to \sigma_3}$
\[ x_2^\circ \sim_{\sigma_2} \tilde{x}_2 \wedge x_3^\circ \sim_{\sigma_3} \tilde{x}_3  \implies v^\circ  \sim_{\rho} \lambda \alpha . v[\tilde{x}_2 \alpha/x_2][\tilde{x}_3 \alpha/x_3]. \]
Therefore given for any $j \in \bp{1,2,3}$ $x_j^\circ$ and $ \tilde{x}_j^{\NN^\NN \to \sigma_j}$ such that $x_j \sim_{\sigma_j} \tilde{x}_j$, we have
\[ v^\circ \sim_{\rho} \lambda \alpha . v[\tilde{x}_2 \alpha/x_2][\tilde{x}_3 \alpha/x_3] \]
which we can plug into the first induction hypothesis to obtain
\[ 
\begin{array}{lcl}
u^\circ v^\circ & \sim_\tau & \lambda \alpha . u[\tilde{x}_1 \alpha/x_1][\tilde{x}_2 \alpha/x_2](v[\tilde{x}_2 \alpha/x_2][\tilde{x_3} \alpha/x_3])) \\[2mm] 
& = & (\lambda \alpha \lambda x_1,x_2,x_3. u v)(\alpha)(\tilde{x}_1 \alpha)(\tilde{x}_2 \alpha)(\tilde{x}_3 \alpha).
\end{array}
\]
\item $t = \Rec^{\eta}$. Without loss of generality we can assume that the recursor has type $\eta = \rho \to \NN$ for some type $\rho$. It is easy to check that $k^\circ \sim_\NN \lambda \alpha . k$, for any variable $k \colon \NN$, where $k^\circ$ is the abbreviation introduced at the end of Definition \ref{circ-def}. \\[1mm]
Assume $x \sim_{\NN} g$, $a \sim_\eta A$, $F \sim_{\NN \to \eta \to \eta} \psi$, $v \sim_\rho V$. We must show that
\[ (\Rec^\eta)^\circ(a,F)(x)(v) \sim_\NN \lambda \alpha. \Rec^\eta(A \alpha, \psi \alpha)(g \alpha)(V \alpha). \]
Or, unfolding the definition of $(\Rec^\eta)^\circ$, that
\[ \apc{\lambda \alpha . \Val_{r[\Val_x(\alpha)]}(\alpha)}{{\sf B}} \sim_\NN \lambda \alpha. \Rec^\eta(A \alpha, \psi \alpha)(g \alpha)(V \alpha) \]
where $r[n]$ and ${\sf B}$ are as in Definition \ref{circ-def}. Again we note that the premise $x \sim_\NN g$ implies that $\Val_{x} = g$ and $(\dagger) \, \GeneralEq{S_x}(\BRSec_{x})$, for a bar $S_x$ securing $g$. \\[1mm]
%
{\bf Claim 1}. For all $n^\NN$, $\Rec^{\eta^\circ}(a, \lambda k^\NN . F(k^\circ))(n) \sim_\eta \lambda \alpha. \Rec^\eta(A \alpha, \psi \alpha)(n)$. \\[1mm]
{\bf Proof}. By induction on $n$. If $n=0$, since $a \sim_{\eta} A$, then 
\begin{align*}
\Rec^{\eta^\circ}(a,  \lambda k^\NN . F(k^\circ))(0) \eqdef  a \sim_\eta A  \eqdef \lambda \alpha. \Rec^\eta(A \alpha,\psi \alpha)(0). 
\end{align*}
For $n>0$, by induction hypothesis we have,
\[ \Rec^{\eta^\circ}(a, \lambda k^\NN . F(k^\circ))(n-1) \sim_\eta \lambda \alpha. \Rec^{\eta}(A \alpha, \psi \alpha)(n-1) \]
Since $F \sim_{\NN \to \eta \to \eta} \psi$ and $(n-1)^\circ \sim_{\NN} \lambda \alpha . n - 1$, we have that for all $b \sim_\eta B$
\[ F((n-1)^\circ,b) \sim_\eta \lambda \alpha . \psi(\alpha)(n - 1, B \alpha). \]
Hence: 
\begin{align*}
\Rec^{\eta^\circ}(a,  \lambda k^\NN . F(k^\circ))(n) 
	&\eqdef  F((n-1)^\circ,\Rec^{\eta^\circ}(a, \lambda k^\NN . F(k^\circ))(n-1)) \\
	& \sim_\eta  \lambda \alpha . \psi(\alpha)(n - 1, \Rec^\eta(A \alpha, \psi \alpha)(n-1))\\
	&\eqdef \lambda \alpha. \Rec^\eta(A \alpha,\psi \alpha)(n). 
\end{align*}
This concludes the proof of the first claim. \\[2mm]
{\bf Claim 2}. For all $n^\NN$ 
\[ r[n] \sim_\NN \lambda \alpha . \Rec^\eta(A \alpha, \psi \alpha)(n)(V \alpha). \]
{\bf Proof}. Immediate from Claim 1 and the assumption $v \sim_\rho V$. \\[1mm]
Claim 2 in particular implies (by the definition of $\sim_\NN$) that for all $n^\NN$
\[ \Val_{r[n]}(\alpha)  = \Rec(A \alpha, \psi \alpha)(n)(V \alpha) \]
and $(\ddagger) \, \GeneralEq{S}(\BRSec_{r[n]})$, for some bar $S$ securing $\lambda \alpha . \Rec^\eta(A \alpha, \psi \alpha)(n)(V \alpha)$. By countable choice $\AC_0$ we have a sequence of bars $(S_n)_{n \in \NN}$. Taking $n = g \alpha$ we have
\begin{itemize}
	\item[] $(i)~\Val_{r[g \alpha]}(\alpha) = \Rec^\eta(A \alpha, \psi \alpha)(g \alpha)(V \alpha)$ 
\end{itemize}
Let $S(s) \eqdef S_x(s) \wedge S_{g \hat s}(s)$. We also have that
\begin{itemize}
	\item[] $(ii)~S$ is a bar securing $\lambda \alpha. \Rec^\eta(A \alpha, \psi \alpha)(g \alpha)(V \alpha)$ 
\end{itemize}
Indeed, since $S(s)$ implies that both $g$ and $\lambda \alpha . \Rec^\eta(A \alpha, \psi \alpha)(g \hat s)(V \alpha)$ are secure, which implies that $\lambda \alpha. \Rec^\eta(A \alpha, \psi \alpha)(g \alpha)(V \alpha)$ is secure. \\[1mm]
{\bf Claim 3}. $\GeneralEq{S}(\lambda G, H . \BRSec_x(\lambda s' . \BRSec_{r[g(\widehat{s'})]}(G,H)(s'),H)).$ \\[1mm]
{\bf Proof}. Fix $G, H$ and $s$. We consider two cases: \\[1mm]
If $S(s)$, then $S_x(s)$ and $S_{g \hat s}(s)$ also hold. Hence, 
\[
\begin{array}{lcl}
\BRSec_x(\lambda s' . \BRSec_{r[g(\widehat{s'})]}(G,H)(s'),H)(s) 
	& \stackrel{(\dagger)}{=} & \BRSec_{r[g(\hat{s})]}(G,H)(s) \\
 	& \stackrel{(\ddagger)}{=} & G(s).
\end{array}
\]
If $\neg S(s)$, then either $\neg S_x(s)$ or $\neg S_{g \hat s}(s)$. We consider two cases: \\[1mm]
If $S_x(s)$ then $\neg S_{g \hat s}(s)$. Then, using that $S_x(s)$ implies both $(*) \, g \hat s = g(\widehat{s * w})$ and $(**) \, S_x(s*w)$,
\[
\begin{array}{lcl}
\BRSec_x(\lambda s' . \BRSec_{r[g(\widehat{s'})]}(G,H)(s'),H)(s) 
	& \stackrel{(\dagger)}{=} & \BRSec_{r[g(\hat{s})]}(G,H)(s)) \\
	& \stackrel{(\ddagger)}{=} & H(s, \lambda w . \BRSec_{r[g(\hat{s})]}(G,H)(s*w)) \\
	& \stackrel{(*)}{=} & H(s, \lambda w . \BRSec_{r[g(\widehat{s*w})]}(G,H)(s*w)) \\
 	& \stackrel{(\dagger, **)}{=} & H(s, \lambda w . \BRSec_x(\lambda s' . \BRSec_{r[g(\widehat{s'})]}(G,H)(s'),H)(s*w)).
\end{array}
\]
Finally, if $\neg S_x(s)$ then the result follows directly by $(\dagger)$.
\end{itemize}
\end{proof}

By combining Theorems \ref{secure-br-thm} and \ref{main-theorem} we obtain:

\begin{corollary}[$\HAomega$] \label{main-cor} Let $\tau = \NN$ and $t \colon \NN$ be a $\T$-term with only $\alpha \colon \tau^\NN$ as free variable. Then $\SpectorEq(\Phi^t(B_{t^\circ}), Y = t)$. Moreover, if $t \in \T_i$ then $\Phi^t(B_{t^\circ}) \in \T_j$, where $j = 2 + \max\{1, \level{\sigma}\} + i$.
\end{corollary}

\begin{proof} By Theorem \ref{main-theorem}, we have $\GeneralEq{S}(\BRSec_{t^\circ})$ for a bar predicate $S$ securing $\lambda \alpha . t[\alpha]$. By Theorem \ref{secure-br-thm} it then follows that $\SpectorEq(\Phi^t(\BRSec_{t^\circ}), Y = t)$. It remains to notice that if $t$ uses a recursor of type $\eta$ then
$t^\circ$ uses a recursor of type $\eta^\circ$. Hence, if $i = \level{\eta}$, by Lemma \ref{type-level-lemma} we have that $\level{\eta^\circ} = 2 + \max \{ 1 + \level{\tau}, \level{\sigma} \} + \level{\eta}$. Since $\level{\tau} = 0$ and $\level{\eta} = i$ we are done. Although we have used countable choice $\AC_0$ in the proof of Theorem \ref{main-theorem}, by modified realizability we can eliminate it here since this corollary is purely universal, so the verification proof that $\Phi^t(B_{t^\circ})$ is a $\T$-definition of bar recursion for $Y = t$ can actually be carried out within $\HAomega$.
\end{proof}

\begin{remark} Note that our construction is parametric in $G$ and $H$, in the sense that we do not require $G$ and $H$ to be $\T$-definable. But once we consider concrete $\T$ terms $Y, G$ and $H$, we get as a corollary Schwichtenberg's result that the functional $\lambda s. \BR(G,H,Y)(s)$ is also $\T$-definable. Unfortunately our construction might not give the ``optimal'' $\T$-definition of $\lambda s. \BR(G,H,Y)(s)$. Indeed, when $Y, G$ and $H$ are in $\T_0$ we obtain a definition of $\lambda s. \BR(G,H,Y)(s)$ in $\T_3$. Howard's analysis \cite{Howard2} suggests that in such cases a definition $\lambda s. \BR(G,H,Y)(s)$ already in $\T_1$ exists. This seems to be the price we need to pay for having a more general construction that works uniformly in $G$ and $H$.
\end{remark}

\begin{remark} Our original motivation for this work started with our bar-recursive bound \cite{BRBOUND} for the Termination Theorem by Podelski and Rybalchenko \cite{Podelski}. The Termination Theorem characterizes the termination of transition-based programs as a properties of well-founded relations. Its classical proof requires Ramsey's Theorem for pairs \cite{Ramsey}. By using Schwichtenberg's result, we proved that under certain hypotheses our bound is in system $\T$. By applying the main construction from this paper we can obtain explicit constructions of the bounds in system $\T$.
\end{remark}

\subsection{Illustrative Example}

Corollary \ref{main-cor} is a generalization of the result obtained by Schwichtenberg in \cite{Schwichtenberg}, but note that our construction is much more explicit, and one we can easily replace bar recursive definitions by their equivalent system $\T$ ones (under the conditions of Schwichtenberg's result). 

Let us go back to the example alluded to in the introduction, i.e. Spector's bar recursion for $t[\alpha] = \Rec^\NN(0, \lambda k.\alpha)(\alpha(0))$, where $\lambda k . \alpha$ is ignoring the first argument $k$ so that $$t[\alpha] = \alpha(\alpha(\ldots(\alpha(0))\ldots)$$ with $\alpha(0)$ applications of $\alpha$. In order to work out the $\T$-definition of the bar recursive functional $\lambda G, H, s . \BR^{\NN, \NN}(G, H, \lambda \alpha . t[\alpha])(s)$ we first calculate $\BRSec_{t^\circ}$,
\[
\BRSec_{(t[\alpha])^\circ}(G,H)(s) 
	= \BRSec_{(\alpha(0))^\circ}(\lambda s'. \BRSec_{r[\Val_{(\alpha(0))^\circ} (\widehat{s'})]}(G,H)(s'),H)(s) 
	= \BRSec_{(\alpha(0))^\circ}(\lambda s'. \BRSec_{r[\widehat{s'}(0)]}(G,H)(s'),H)(s)
\]
where $r[n] = \Rec^{\NN^\circ}(0^\circ, \lambda k.\alpha^\circ)(n)$. By Corollary \ref{main-cor}, $\GeneralEq{S}(\BRSec_{(t[\alpha])^\circ})$ for some bar $S$ which secures $\lambda \alpha . t[\alpha]$. Hence, $\BR^{\NN,\NN}(G,H,\lambda \alpha. t[\alpha])$ can be $\T$-defined as
\eqleft{
\BR^{\NN,\NN}(G,H,\lambda \alpha. t[\alpha]) = \Phi^{\lambda \alpha . t[\alpha]}(\lambda G, H. \BRSec_{t^\circ}(G,H))(G, H)
}
with $\Phi^{\lambda \alpha . t[\alpha]}$ as in the proof of Theorem \ref{secure-br-thm}, i.e.
\eqleft{
\BR^{\NN,\NN}(G,H,\lambda \alpha. t[\alpha]) = \BRSec_{t^\circ}(\lambda s'.\Psi(t[\widehat{s'}])(G, H)(s'), \mathcal{H}^{\lambda \alpha . t[\alpha]}(G, H) )
}
where
\eqleft{
{\mathcal H}^Y(G,H)(s)(f^{\NN \to \NN}) \eqdef
\begin{cases}
    G(s) & {\rm if} \; Y(\hat s) < |s| \\[2mm]
    H(s)(f) & {\rm otherwise}.
\end{cases}
}

\subsection{The Case $\tau = \NN \to \NN$}
\label{sec-case-fct}

We now discuss how to extend the construction given in Definition \ref{circ-def}, and the proof of Theorem \ref{main-theorem}, so that Corollary \ref{main-cor} also holds when $\tau = \NN \to \NN$. In this case $\alpha$ has type $\NN \to (\NN \to \NN)$. First, in Definition \ref{circ-def}, when $\tau=\NN \to \NN$ we modify the definition of $\alpha^\circ$ as
\[ \alpha^\circ  \eqdef  \lambda x^{\NN^{\circ}}  y^{\NN^{\circ}} . \apc{\Val}{\BRSec} \]
where
\begin{itemize}
	\item $\Val(\alpha) \eqdef \alpha(\Val_x(\alpha))(\Val_y(\alpha))$,
	\item $\BRSec(G,H)(s) \eqdef \BRSec_y(\BRSec_x(\lambda s'.\Psi(\max\bp{\Val_x(\widehat{s'}), \Val_y(\widehat{s'})})(G,H)(s'),H),H)(s)$.
\end{itemize}
We also need to modify the proof of Theorem \ref{main-theorem} in the place where the case $\alpha$ is treated. Let $x^\circ \sim g$ and $y^\circ  \sim h$. This implies that $(\dagger)$ $\GeneralEq{S_x}(\BRSec_x)$ for a bar $S_x$ securing $g$, and $(\ddagger)$ $\GeneralEq{S_y}(\BRSec_y)$ for a bar $S_y$ securing $h$.  Define the predicate:
\[S(s) \eqdef S_x(s) \wedge S_y(s) \wedge \max\bp{\Val_x(\hat{s}), \Val_y(\hat{s})} < |s|.\] 
That $S(s)$ is a bar follows directly from the assumptions that $S_x$ and $S_y$ are bars. We show that 
$\GeneralEq{S}(\BRSec)$.
Consider two cases:

\noindent If  $S(s)$ holds, then  $S_x(s) \wedge S_y(s) \wedge \max\bp{g(\hat s), h(\hat{s}) }< |s|$. In this case we trivially have
\[
\begin{array}{lcl}
\BRSec(G,H)(s)
	 & \stackrel{(\ddagger)}{=} & \BRSec_x(\lambda s'.\Psi(\max\bp{g(\widehat{s'}), h(\widehat{s'})})(G,H)(s'),H)(s) \\[1mm]
	 & \stackrel{(\dagger)}{=} & \Psi(\max\bp{g(\hat s), h(\hat s)})(G,H)(s) \\[1mm]
	 & = & G(s). 
\end{array}
\]
If  $\neg S(s)$ holds,  we consider three cases:

\noindent If $\neg S_y(s)$ then
	\begin{align*}
    	\BRSec(G,H)(s)
    	    &\eqdef \BRSec_y(\BRSec_x(\lambda s'.\Psi(\max\bp{g(\widehat{s'}), h(\widehat{s'})})(G,H)(s'),H),H)(s)\\
        	&\stackrel{(\ddagger)}{=} H(s)( \lambda z.\BRSec_y(\BRSec_x(\lambda s'.\Psi(\max\bp{g(\widehat{s'}), h(\widehat{s'})})(G,H)(s'),H),H)(s*z))\\
        	&\eqdef H(s)(\lambda z . \BRSec(G,H)(s * z)).
	\end{align*}
	
\noindent If $S_y(s)$ but $\neg S_x(s)$ then, by monotonicity we have also $S_y(s*z)$ for every $z$. Thus: 
		\begin{align*}
		\BRSec(G)(s)
	    &\eqdef \BRSec_y(\BRSec_x(\lambda s'.\Psi(\max\bp{g(\widehat{s'}), h(\widehat{s'})})(G,H)(s'),H),H)(s)\\
		&\stackrel{(\ddagger)}{=}  \BRSec_x(\lambda s'.\Psi(\max\bp{g(\widehat{s'}), h(\widehat{s'})})(G,H)(s'),H)(s)\\
		&\stackrel{(\dagger)}{=} H(s)(\lambda z . \BRSec_x(\lambda s'.\Psi(\max\bp{g(\widehat{s'}), h(\widehat{s'})})(G,H)(s'),H)(s * z)) \\
		&\stackrel{(\ddagger)}{=} H(s)( \lambda z.\BRSec_y(\BRSec_x(\lambda s'.\Psi(\max\bp{g(\widehat{s'}), h(\widehat{s'})})(G,H)(s'),H),H)(s*z))\\
		&\eqdef H(s)(\lambda z . \BRSec(G,H)(s * z)).
		\end{align*}
		
\noindent If $S_y(s)$ and $S_x(s)$ and $\max\bp{g(\hat s), h(\hat s)}\geq |s|$. From $S_y(s)$ and $S_x(s)$ we have $(*)$ $g(\widehat{s*z})= g(\hat s) \wedge h(\widehat{s*z})= h(\hat s)$ for every $z$. Moreover, by monotonicity we have also $S_y(s*z)$ and $S_x(s* z)$. 
	\begin{align*}
	\BRSec(G,H)(s)
	&\eqdef \BRSec_y(\BRSec_x(\lambda s'.\Psi(\max\bp{g(\widehat{s'}), h(\widehat{s'})})(G,H)(s'),H),H)(s)\\
	&\stackrel{(\ddagger)}{=}  \BRSec_x(\lambda s'.\Psi(\max\bp{g(\widehat{s'}), h(\widehat{s'})})(G,H)(s'),H)(s)\\
	&\stackrel{(\dagger)}{=}   \Psi(\max\bp{g(\hat s), h(\hat s)})(G, H)(s) \\
	&\eqdef H(s)(\lambda z . \Psi(\max\bp{g(\hat{s}), h(\hat{s})})(G, H)(s * z)) \\
	&\stackrel{(*)}{=}   H(s)(\lambda z . \Psi(\max\bp{g(\widehat{s*z}), h(\widehat{s*z})})(G, H)(s * z)) \\
	&\stackrel{(\dagger)}{=}  H(s)( \lambda z.\BRSec_x(\lambda s'.\Psi(\max\bp{g(\widehat{s'}), h(\widehat{s'})})(G,H)(s'),H)(s*z)) \\
	&\stackrel{(\ddagger)}{=} H(s)( \lambda z.\BRSec_y(\BRSec_x(\lambda s'.\Psi(\max\bp{g(\widehat{s'}), h(\widehat{s'})})(G,H)(s'),H),H)(s*z)) \\
	&\eqdef H(s)(\lambda z . \BRSec(G,H)(s * z)).
	\end{align*}

\vspace{3mm}
\noindent \textbf{Acknowledgements.} The authors are grateful to Stefano Berardi, Ulrich Kohlenbach, Helmut Schwichtenberg and the reviewer for various useful comments and suggestions.

\bibliographystyle{asl}
\bibliography{biblio}	

\end{document}